\documentclass [14pt]{amsart}
\usepackage{amsmath}
\usepackage{amssymb}
\usepackage{amscd}
\usepackage{epsfig}
\usepackage{amsxtra}
\usepackage{mathabx}
\usepackage{MnSymbol}
\usepackage{scalerel,stackengine}
\usepackage{xcolor}
\usepackage{epsfig}
\usepackage{tikz-cd}
\usepackage{hyperref}
\usepackage{tcolorbox}

\newcommand{\dens}{\mbox{\rm dens}}
\newcommand{\RR}{{\mathbb R}}
\newcommand{\PP}{{\mathbb P}}

\newcommand{\ZZ}{{\mathbb Z}}

\newcommand{\TT}{\mathbb T}
\newcommand{\NN}{\mathbb N}

\newcommand{\vol}{\mbox{vol}}
\newcommand{\cA}{{\mathcal A}}
\newcommand{\cL}{{\mathcal L}}
\newcommand{\cB}{{\mathcal B}}

\newcommand{\oplam}{\mbox{\Large $\curlywedge$}}

\newcommand{\ud}{\overline{\mbox{dens}}}
\newcommand{\dd}{\mbox{d}}

\newcommand{\Cc}{C_{\mathsf{c}}}

 \newtheorem{theorem}{Theorem}[section]
 \newtheorem{lemma}[theorem]{Lemma}
 \newtheorem{proposition}[theorem]{Proposition}
 \newtheorem{corollary}[theorem]{Corollary}
 \newtheorem{fact}[theorem]{Fact}
 \newtheorem{conj}[theorem]{Conjecture}
 
 \newtheorem{definition}[theorem]{Definition}
 \newtheorem{example}[theorem]{Example}
  \newtheorem{remark}[theorem]{Remark}

\begin{document}
\title{On arithmetic progressions in model sets}
\author{Anna Klick}
\address{Department of Mathematical Sciences, MacEwan University \newline
\hspace*{\parindent} 10700 -- 104 Avenue, Edmonton, AB, T5J 4S2, Canada}
\email{klicka@mymacewan.ca}

\author{Nicolae Strungaru}
\address{Department of Mathematical Sciences, MacEwan University \newline
\hspace*{\parindent} 10700 -- 104 Avenue, Edmonton, AB, T5J 4S2, Canada\\
and \\
Institute of Mathematics ``Simon Stoilow''\newline
\hspace*{\parindent}Bucharest, Romania}
\email{strungarun@macewan.ca}
\urladdr{http://academic.macewan.ca/strungarun/}

\author{Adi Tcaciuc}
\address{Department of Mathematical Sciences, MacEwan University \newline
\hspace*{\parindent} 10700 -- 104 Avenue, Edmonton, AB, T5J 4S2, Canada}
\email{tcaciuca@macewan.ca}

\begin{abstract} In this project we show the existence of arbitrary length arithmetic progressions in model sets and Meyer sets in the Euclidean $d$-space. We prove a van der Waerden type theorem for Meyer sets. We show that pure point subsets of Meyer sets with positive density and pure point diffraction contain arithmetic progressions of arbitrary length.
\end{abstract}

\maketitle

\section{Introduction}

The discovery of quasicrystals in 1980's \cite{She} has triggered increased interest in structures with long range aperiodic order, usually shown via a clear diffraction pattern.
The best mathematical models for quasicrystals are model sets. Introduced by Meyer in 1970's \cite{Meyer}, model sets and their subsets, called now Meyer sets, have been popularized in the area of aperiodic order by Lagarias \cite{LAG,LAG1} and Moody \cite{Moody,MOO}. They are constructed via so-called \emph{cut-and project-schemes (CPS)}, which involve cutting a piece of a higher dimensional lattice, bounded around the real space, and projecting the points into the real space (see Def.~\ref{Cps} below for a precise definition).

Model sets with regular windows show long-range order via a clear pure point diffraction spectrum \cite{Hof,Martin2,BM,CR,LR,TAO,CRS}, which can even be traced to the underlying lattice \cite{CRS}. Meyer sets show long range order, in the form of a large pure point diffraction spectrum \cite{NS1,NS2,NS5,NS11,NS12}, which is highly ordered \cite{NS2,NS11}.

The goal of this project is to show the existence of arbitrary long arithmetic progressions in model sets and Meyer sets, results which are of the same natures as classical theorems for subsets of $\ZZ$. This is evidence of high coherence inside Meyer sets, which is a leftover of the lattice in the underlying CPS.

\smallskip

The existence of arithmetic progressions of arbitrary length in subsets of $\ZZ$ is well studied.
In 1927, van der Waerden proved \cite{vdW} the following theorem:

\setcounter{theorem}{7}
\setcounter{section}{2}
\begin{theorem}[van der Waerden's Theorem] Given any natural numbers $k,r$, there exists a number $W(r, k)$, such that for any colouring of  $\mathbb Z$ with $r$ colors, and for any $N \geq  W(r,k)$, the set $\{1, 2, 3, \dots ,N\}$ contains a monochromatic arithmetic progression of length $k$.
\end{theorem}
Intuitively, this theorem says that if we split the integers into $r$ disjoint sets, at least one of the sets will have arithmetic progressions of arbitrary length. Moreover, there exists a bound $W(r,k)$ on how far one needs to go to find such an arithmetic sequence, which depends on $k$ and $r$ but it is independent of the splitting.

In 1975, Szemer\'edi extended the result, proving the following well known conjecture of Erd\"os and Tur\'an \cite{ET}:
\begin{theorem}[Szemer\'edi's Theorem]\cite{sze} Let $\Lambda \subset \NN$ be a subset with the property that
\begin{displaymath}
\ud(\Lambda):= \limsup_n \frac{\sharp \{ 1 , 2 , 3, ..., n \} \cap \Lambda}{n} >0 \,.
\end{displaymath}

 Then $\Lambda$ contains arithmetic progressions of arbitrary length.
\end{theorem}
\setcounter{theorem}{1}
\setcounter{section}{1}
It is easy to see that Szemer\'edi's theorem implies van der Waerden's theorem; indeed any finite partition of $\NN$ contains a set with positive density. It is also easy to construct subsets of $\NN$ of zero density which don't contain arithmetic progressions of large length. For example, the set $\Lambda=\{2^n: n \in \NN \}$ cannot contain arithmetic sequences of length $3$. Nevertheless, under suitable extra conditions, one can still hope to find arbitrary long arithmetic progressions in sets of zero upper density. In 2004, Green and Tao proved the following landmark result, solving the long-standing conjecture that the set of primes contains arbitrarily long arithmetic progressions.

\textbf{Theorem} (Green--Tao Theorem)\cite{GT}{\em Let $\Lambda \subset \PP$ be a subset of the primes with the property that
\begin{displaymath}
 \limsup_n \frac{\sharp \{ 1 , 2 , 3, ..., n \} \cap \Lambda}{\pi(n)} >0 \,.
\end{displaymath}
where $\pi(N)$ is the prime counting function. Then $\Lambda$ contains arithmetic progressions of arbitrary length. }

A stronger version of Szemer\'edi's Theorem, which would imply the Green--Tao theorem, is the following famous conjecture of Erd\"os:

\begin{conj} Let $\Lambda \subset \NN$ be a set such that
\begin{displaymath}
\sum_{n \in \Lambda} \frac{1}{n} = \infty \,.
\end{displaymath}
Then $\Lambda$ contains arithmetic progressions of arbitrary length.
\end{conj}

\smallskip

In all the results above, the existence of the arithmetic progressions can be traced to the high order present in $\ZZ$ and $\NN$. The goal of our project is to extend these results to highly ordered aperiodic structures, by studying the existence of arithmetic progressions of arbitrary length in model sets and Meyer sets. This is a natural generalization, as model sets are usually considered the natural candidate for "aperiodic" lattices, and Meyer sets are the natural candidate for "aperiodic" lattice subsets.

Given a set $\Lambda \subseteq \RR^d$, by an arithmetic progression of length $k$ we understand a sequence $a_1,a_2,...,a_k \in \Lambda$ with the property that there exists some $r \neq 0$ so that for all $1 \leq j \leq k$ we have $a_{j} =a_{1}+(j-1)r$. This is equivalent to
\begin{displaymath}
a_2-a_1=a_3-a_2=...=a_k-a_{k-1}\neq 0 \,.
\end{displaymath}

Given a model set $\oplam(W)$ (see Definition~\ref{def:model set} below), we first prove the following van der Waerden type result:

\setcounter{theorem}{3}
\setcounter{section}{4}
\begin{theorem}[van der Waerden's theorem for model sets] Let $\Lambda \subset \RR^d$ be a model set.  Then, for any positive integers $r$ and $k$, there is some $R>0$ such that for any colouring of  $\Lambda$ with $r$ different colors, and for any $x\in \RR^d$,  the set $\Lambda \cap B_R(x)$ contains a monochromatic arithmetic progression of length $k$.
\end{theorem}
\setcounter{theorem}{1}
\setcounter{section}{1}

We then  extend this result to Meyer sets (see Definition~\ref{def Meyer} below).
\setcounter{theorem}{1}
\setcounter{section}{5}
\begin{theorem}[van der Waerden's theorem for Meyer sets] Let $\Lambda\subset \RR^d$  be a Meyer set. Then, for any positive integers $r$ and $k$, there is some $R>0$ such that for any colouring of  $\Lambda$ with $r$ different colors, and for any $x\in \RR^d$,  the set $\Lambda \cap B_R(x)$ contains a monochromatic arithmetic progression of length $k$.
\end{theorem}
\setcounter{theorem}{1}
\setcounter{section}{1}

We should emphasize here that, while Theorem~\ref{TV5} implies Theorem~\ref{TV2}, the proof of Theorem~\ref{TV5} uses Theorem~\ref{TV2}, so proving Theorem~\ref{TV2} is necessary for our approach.

This generalizes some partial results in this direction obtained in \cite{Nag,LlW}.

We complete the paper by showing that pure point diffractive subsets of Meyer sets of positive density contain arbitrary long arithmetic progressions. In particular, weak model sets of maximal density have this property.

\section{Preliminaries}

Recall first that a set $\Lambda \subseteq \RR^d$ is called \textbf{relatively dense} if there exists some $R>0$ such that
\begin{displaymath}
\Lambda+B_R(0) = \RR^d \,.
\end{displaymath}
In this case, if we want to emphasize the constant $R$, we will say that $\Lambda$ is $R$-relatively dense.

Here, for two sets $A,B \subset G$ we denote by $A \pm B$ the Minkowski sum and/or difference
\begin{displaymath}
A \pm B := \{ a \pm b : a \in A, b\in B \} \,.
\end{displaymath}
Equivalently, $\Lambda$ is $R$-relatively dense if for each $x \in \RR^d$ there exists some $y \in \Lambda$ with $d(x,y) \leq R$.

A set $\Lambda \subseteq \RR^d$ is called \textbf{uniformly discrete} if there exists some $r>0$ such that, for all $x, y \in \Lambda$ with $x \neq y$ we have $d(x,y) \geq r$.

\smallskip

A set $\Lambda \subset \RR^d$ is called \textbf{locally finite} if, for each $R>0$ the set $\Lambda \cap B_R(0)$ is finite. This is equivalent to $\Lambda$ being closed and discrete in $\RR^d$.

\smallskip

\begin{definition} A finite sequence $a_1,a_2,...,a_n$ in $\RR^d$ is called an \textbf{arithmetic progression} if
\begin{displaymath}
a_2-a_1=a_3-a_2=...=a_n-a_{n-1} \,.
\end{displaymath}
\end{definition}

\begin{remark} $a_1,a_2,...,a_n$ is an arithmetic progression in $\RR^d$ if and only if, there exists some $s,t$ such that, for all $1 \leq k \leq n$ we have
\begin{displaymath}
a_k= s+(k-1)t \,.
\end{displaymath}
\end{remark}

\smallskip

Next, we review the notions of Cut and Project scheme, model set and Meyer set. For a general overview of these topics we recommend the monographs \cite{TAO,TAO2}, as well as \cite{Meyer,MOO,Moody,LAG1,Martin2,CR,LR,NS1,NS2,NS5,NS11,NS12}.

\begin{definition}\label{Cps} By a \textbf{Cut and Project scheme}, or simply (CPS), we understand a triple $(\RR^d, H, \cL)$ consisting of $\RR^d$, a locally compact Abelian group (LCAG)  $H$, together with a lattice (i.e. a discrete co-compact subgroup) $\cL \subset \RR^d \times H$, with the following two properties:
\begin{itemize}
    \item{} The restriction $\pi^{\RR^d}|_{\cL}$ of the canonical projection $\pi^{\RR^d}: \RR^d \times H \to \RR^d$ to $\cL$ is a one-to-one function.
    \item{} The image $\pi^H(\cL)$ of the $\cL$ under the canonical projection $\pi^H: \RR^d \times H \to H$ is dense in $H$.
\end{itemize}
\end{definition}

Given a CPS $(\RR^d, H, \cL)$ we usually denote by
\begin{displaymath}
L:= \pi^{\RR^d}( \cL) \,.
\end{displaymath}
$L$ is a subgroup of $\RR^d$, which is typically dense in $\RR^d$.

The first condition in the definition of a CPS implies that we can define a mapping $\star : L \to H$, called the \textbf{$\star$-mapping} as
\begin{displaymath}
\star= \pi^H \circ ( \pi^{\RR^d}|_{\cL})^{-1} \,.
\end{displaymath}
Then, we can reparametrize $\cL$ as
\begin{displaymath}
\cL= \{ (x,x^\star) : x \in L \} \,.
\end{displaymath}
The range of the $\star$-mapping is
\begin{displaymath}
L^\star:= \pi^H(\cL) \,.
\end{displaymath}
 We can summarize a CPS in the following picture.

\smallskip
\begin{center}
\begin{tikzcd}[remember picture]
\RR^d & \arrow[swap]{l}{\pi^{\RR^d}} \RR^d \times H \arrow{r}{\pi^H}& H \\
\pi^{\RR^d}(\cL)&  \arrow[right]{l}{1-1} \cL  \arrow[left]{r}{dense} & H \\
L \arrow[swap]{rr}{\star} &  & L^\star \\
\end{tikzcd}
\end{center}

\begin{tikzpicture}[overlay,remember picture]
\path (\tikzcdmatrixname-1-1) to node[midway,sloped]{$\supset$}
(\tikzcdmatrixname-2-1);
\path (\tikzcdmatrixname-1-2) to node[midway,sloped]{$\supset$}
(\tikzcdmatrixname-2-2);
\path (\tikzcdmatrixname-1-3) to node[midway,sloped]{$=$}
(\tikzcdmatrixname-2-3);
\path (\tikzcdmatrixname-2-1) to node[midway,sloped]{$=$}
(\tikzcdmatrixname-3-1);
\path (\tikzcdmatrixname-2-3) to node[midway,sloped]{$\supset$}
(\tikzcdmatrixname-3-3);
\end{tikzpicture}
\medskip

\begin{definition}\label{def:model set}
Given a CPS $(\RR^d,H,\cL)$ and some subset $W \subset H$, we denote by $\oplam(W)$ its pre-image under the $\star$-mapping, that is
\begin{displaymath}
\oplam(W):=\{ x \in L : x^\star \in W \} = \{ x \in \RR^d : \exists y \in W \, \mbox{ such that } (x,y) \in \cL \} \,.
\end{displaymath}

When $W$ is compact, the set $\oplam(W)$ is called a \textbf{weak model set}.

If $W$ has non-empty interior and compact closure, the set $\oplam(W)$ is called a \textbf{model set}.

An example of a CPS and a model set is included in Fig.~\ref{ms pic}.
\end{definition}

\medskip

Of importance to us is the following result
\begin{lemma}\cite{MOO,Moody}\label{model sets are rel dense} Let $(\RR^d, H, \cL)$ be a CPS and $W \subset H$. If $W$ has non-empty interior, then $\oplam(W)$ is relatively dense.
\end{lemma}

\medskip
Next, we briefly review the concept of Meyer sets. For a more general review we recommend the paper \cite{MOO} (or \cite{NS11} for arbitrary LCAG) .

\begin{definition}\label{def Meyer} A subset $\Lambda \subset \RR^d$ is called a \textbf{Meyer set} if $\Lambda$ is relatively dense, and $\Lambda-\Lambda$ is uniformly discrete.
\end{definition}

Of importance to us is the following characterization of Meyer sets.

\begin{theorem}\label{Meyer characterisation}\cite{MOO,NS11}
Let $\Lambda \subset \RR^d$ be relatively dense. Then, the following are equivalent.
\begin{itemize}
    \item [(i)] $\Lambda$ is a Meyer set.
    \item[(ii)] There exists a model set $\oplam(W)$ such that $\Lambda \subseteq \oplam(W)$.
    \item[(iii)] $\Lambda$ is locally finite and there exists a finite set $F$ such that
    \begin{displaymath}
    \Lambda - \Lambda \subset \Lambda+F \,.
    \end{displaymath}
\end{itemize}
\end{theorem}

\medskip
\subsection{Arithmetic progressions in sets of natural numbers}
In this subsection we  review the van der Vaerden and Szemer\'edi Theorems.

\begin{theorem}[van der Waerden Theorem]\cite{vdW}\label{van der W} Given any natural numbers $k,r$, there exists a number $W(r, k)$, such that, no matter how we color the integers $\mathbb Z$ with $r$ colors, for each $N \geq  W(n,k)$, in the set $\{1, 2, 3, \dots ,N\}$ we can find an arithmetic progression of length $k$ whose elements are all of the same color.
\end{theorem}

\begin{theorem}[Szemer\'edi Theorem]\cite{sze} Let $\Lambda \subset \NN$ be a subset with the property that
\begin{displaymath}
\ud(\Lambda):= \limsup_n \frac{\sharp \{ 1 , 2 , 3, ..., N \} \cap \Lambda}{n} >0 \,.
\end{displaymath}
Then $\Lambda$ contains arithmetic progressions of arbitrary length.
\end{theorem}

\subsection{Diffraction}

We complete the section by reviewing briefly the notion of pure point diffraction, which we will use in Section~\ref{sect:pure}.

Note that in Section~\ref{sect:pure} we will only use the characterization \eqref{eq pp} in Prop.~\ref{P4} below, but we need to introduce the following concepts to be able to introduce Prop.~\ref{P4}. As we only need these concepts for some particular point sets, we restrict to this case and refer the reader to \cite{TAO} for the more general case.

\smallskip

Next, we say that a subset $\Lambda \subset \RR^d$ has \textbf{Finite Local Complexity} (or \textbf{FLC}) if the set $\Lambda- \Lambda$ is locally finite.

\begin{definition} Let $\Lambda \subset \RR^d$ be a set with FLC, and $A_n=[-n,n]^d$.  We say that the \textbf{autocorrelation} $\gamma$ of $\Lambda$ exists with respect to $\cA=\{ A_n \}$, if, for all $z \in \Lambda-\Lambda$, the following limit exists
\begin{displaymath}
\eta(z):= \lim_n \frac{1}{(2n)^d} \sharp \{ (x,y): x,y \in \Lambda \cap A_n, x-y=z \}
\end{displaymath}
In this case, we define
\begin{displaymath}
\gamma:= \sum_{z \in \Lambda - \Lambda} \eta(z) \delta_z \,.
\end{displaymath}
\end{definition}

Given $\Lambda \subset \RR^d$ a set with FLC, the autocorrelation $\gamma$ always exists with respect to some subsequence $\cB$ of $\cA$ \cite{TAO,BL,BM,Martin2}.

\smallskip

Let $\Lambda \subset \RR^d$ be a set with FLC and assume that its autocorrelation $\gamma$ exists with respect to some subsequence $\cB$ of $\cA$. Then, there exists a positive measure $\widehat{\gamma}$ on the $\widehat{\RR^d} \simeq \RR^d$  such that \cite{ARMA1,BF,MoSt}
\begin{equation}\label{diffraction}
\int_{\RR^d} \left| \check{f} \right|^2(t) \dd \widehat{\gamma}(t)= \int_{G} f*\tilde{f}(s) \dd \gamma(s)
\end{equation}
holds for all $f \in \Cc(\RR^d)$, that is continuous compactly supported functions.

Here, for $f,g \in \Cc(\RR^d)$ we use the standard notations
\begin{align*}
   \tilde{f}(x)&=\overline{f(-x)} \\
   \check{f}(t)&= \int_{\RR^d} e^{2 \pi i s \cdot t} f(s) \dd s \\
   f*g(x)&=\int_{\RR^d} f(x-t) g(t) \dd t
\end{align*}

\begin{definition} We call the measure $\widehat{\gamma}$ from \eqref{diffraction} the \textbf{diffraction of $\Lambda$} with respect to $\cB$.

We say that $\Lambda$ is \textbf{pure point diffractive with respect to $\cB$} if the measure $\widehat{\gamma}$ is a pure point measure.
\end{definition}

\smallskip

Next, we review the following metric.

For two uniformly discrete pointsets $\Lambda$ and $\Gamma$, and let $\cB$ be a fixed subsequence of $\cA= \{ A_n \}_n$ where  $A_n=[-n,n]^d$. Define
\begin{displaymath}
d_{\cB}(\Lambda, \Gamma):= \limsup_n \frac{1}{\vol(B_n)} \sharp (\Lambda \Delta \Gamma) \cap B_n=: \overline{\mbox{dens}_{\cB}(\Lambda \Delta \Gamma)} \,.
\end{displaymath}
The topology induced by this metric on the hull of a pointset is called the autocorrelation topology in \cite{BL,BLM}. In \cite{Gouere}, the author refers to this as the Paterson topology, in \cite{MS} is called the statistical coincidence topology, while in \cite{LSS} this is called the mean topology.
Whenever the sequence $\cB$ is clear from context, we will simply denote $d_{\cB}(\Lambda, \Gamma)$.

\smallskip
We will use the following properties of this metric.

\begin{lemma}\cite{MS,LSS}\label{L5} Let $\cB$ be a fixed subsetequence of $\cA$. For each $r>0$, $d_\cB$ defines a translation invariant semi-metric on $\mathcal{UD}_r(\RR^d)$, the set of $r$-uniformly discrete subsets of $\RR^d$.
\end{lemma}

Of importance to us is the following result:
\begin{proposition}\label{P4}\cite[Thm.~5]{BM} Let $\Lambda \subset \RR^d$ be so that $\Lambda- \Lambda$ is uniformly discrete, and let $\cB$ be any subsequence of $A_n=[-n,n]^d$ with respect to which the autocorrelation of $\Lambda$ exists. Then, $\Lambda$ is pure point diffractive with respect to $\cB$ if and only if, for each $\epsilon >0$ the set
\begin{equation}\label{eq pp}
P_\epsilon:= \{ t \in \RR^d : d_{\cB}(\Lambda, t+\Lambda) < \epsilon \}
\end{equation}is relatively dense.
\end{proposition}

For more general versions of Prop~\ref{P4} see \cite{Gouere,LSS}.

\section{Arithmetic progressions in the Fibonacci Model set}\label{sect:fib}

Before looking at the general case, let us first look at the Fibonacci model set. For a detailed overview of this see \cite[Chapter~7]{TAO}.

Consider the following CPS:

\begin{center}
\begin{tikzcd}
\RR & \arrow[swap]{l}{\pi_1} \RR\times \RR \arrow{r}{\pi_2}& \RR \\
L \arrow[hookrightarrow]{u} & \arrow[hookrightarrow]{u} \arrow[right]{l}{1-1} \cL:=\ZZ ( 1 , 1 ) \oplus \ZZ ( \tau , \tau' )  \arrow[left]{r}{\pi_2} & \arrow[hookrightarrow, swap]{u}{\text{dense}}L^\star \\
\ZZ[\tau] \arrow[equal]{u}  \arrow[swap]{rr}{\star} &  & \arrow[equal]{u}\ZZ[\tau] \\
\end{tikzcd}
\end{center}
\smallskip
Here, $\tau=\frac{1+\sqrt{5}}{2}$ and $\tau'=\frac{1-\sqrt{5}}{2}$ and the $\star$ mapping is simply the Galois conjugation
\[
(m+n \tau)^\star=m+n\tau'  \qquad \forall m,n \in \ZZ
\]

\smallskip
With this CPS, the Fibonacci model set is
\begin{displaymath}
\oplam([-1, \tau-1)):=\{ x \in L : x^\star \in [-1, \tau-1) \}\,,
\end{displaymath}
see Figure~\ref{ms pic} below.
\smallskip

With this particular CPS and window, the Fibonacci model set coincides with the left end points of the geometric realization of the substitution rule
\begin{align*}
a & \to ab \\
b&\to a  \,,\\
\end{align*}
see \cite[Chapter~7]{TAO}.

\medskip

\begin{figure}[ht]
  \centering
\begin{tikzpicture}[>=latex,scale=.5]
\clip (-5,-5) rectangle (11cm,11cm); 
\draw [fill=gray!20] (-8,-1) rectangle (11,.62);
    \foreach \x in {-13,-12,...,13}{
    \foreach \y in {-13,-12,...,13}{
    \node[draw,circle,inner sep=1pt,fill] at (\x+1.62*\y,\x-.62*\y) {}; 
    }
    }
\draw[->,thick](-5,0) --(11,0);
\draw[->,thick](0,-5) --(0,11);
\draw[line width=0.8mm,  blue](0,-1)--(0,.62);
\draw[thick](-1,-1) --(-1,0);
\draw[thick](-2.62,-.38) --(-2.62,0);
\draw[thick](-4.24,.24) --(-4.24,0);
\draw[thick](1.62,-.62) --(1.62,0);
\draw[thick](2.62,.38) --(2.62,0);
\draw[thick](4.24,-.24) --(4.24,0);
\draw[thick](5.86,-.86) --(5.86,0);
\draw[thick](6.86,.14) --(6.86,0);
\draw[thick](8.48,-.48) --(8.48,0);
\draw[thick](9.48,.52) --(9.48,0);
\node[draw,circle,inner sep=1.5pt,fill,red] at (0,0) {};
\node[draw,circle,inner sep=1.5pt,fill,red] at (-1,0) {};
\node[draw,circle,inner sep=1.5pt,fill,red] at (-2.62,0) {};
\node[draw,circle,inner sep=1.5pt,fill,red] at (-4.24,0) {};
\node[draw,circle,inner sep=1.5pt,fill,red] at (1.62,0) {};
\node[draw,circle,inner sep=1.5pt,fill,red] at (2.62,0) {};
\node[draw,circle,inner sep=1.5pt,fill,red] at (4.24,0) {};
\node[draw,circle,inner sep=1.5pt,fill,red] at (5.86,0) {};
\node[draw,circle,inner sep=1.5pt,fill,red] at (6.86,0) {};
\node[draw,circle,inner sep=1.5pt,fill,red] at (8.48,0) {};
\node[draw,circle,inner sep=1.5pt,fill,red] at (9.48,0) {};
\end{tikzpicture}
  \caption{Fibonacci Model set. The window is in blue and the model set in red.}
  \label{ms pic}
\end{figure}
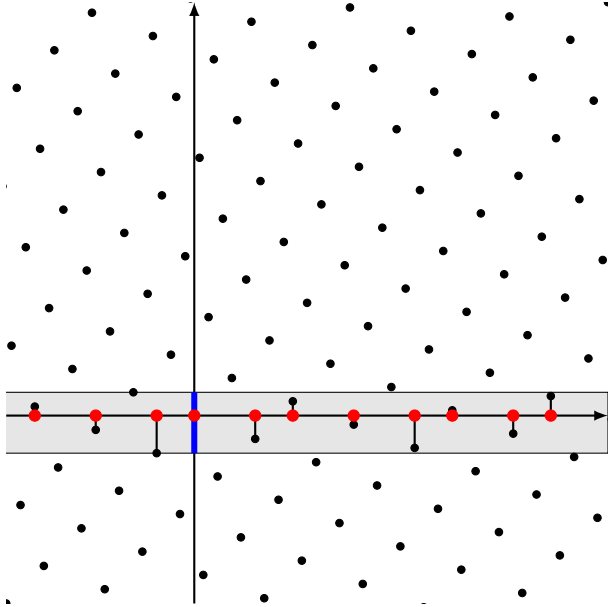

\medskip
Now, we can prove the following result.

\begin{lemma} Let $n \in \NN$, and let $\Lambda=\oplam([-1, \tau-1))$ be the Fibonacci model set. Then, there exist $s, t \in \RR$ such that
\begin{displaymath}
s, s+t, s+2t,..., s+nt \in \Lambda \,.
\end{displaymath}
\end{lemma}
\begin{proof}
Pick some $t \in \oplam([0, \frac{\tau-1}{n}))$. Then, by construction $0,t^\star,2t^\star,..., nt^\star \in [-1, \tau-1)$ and hence
\begin{displaymath}
0, 0+t, 0+2t,..., 0+nt \in \Lambda \,.
\end{displaymath}
\end{proof}

\medskip

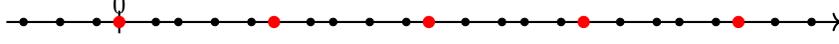
\begin{figure}[ht]
  \centering
\begin{tikzpicture}[scale=0.3]
\draw[->,thick](-5,0) --(32,0);
\draw[thick](0,-.5) --(0,.5);
\draw (0,0) node[anchor=south] {0};
\node[draw,circle,inner sep=1.5pt,fill,red] at (0,0) {};
\node[draw,circle,inner sep=1pt,fill] at (1*1.62,0) {};
\node[draw,circle,inner sep=1pt,fill] at (1*1.62+1,0) {};
\node[draw,circle,inner sep=1pt,fill] at (2*1.62+1,0) {};
\node[draw,circle,inner sep=1pt,fill] at (3*1.62+1,0) {};
\node[draw,circle,inner sep=1.5pt,fill,red] at (3*1.62+2,0) {};
\node[draw,circle,inner sep=1pt,fill] at (4*1.62+2,0) {};
\node[draw,circle,inner sep=1pt,fill] at (4*1.62+3,0) {};
\node[draw,circle,inner sep=1pt,fill] at (5*1.62+3,0) {};
\node[draw,circle,inner sep=1pt,fill] at (6*1.62+3,0) {};
\node[draw,circle,inner sep=1.5pt,fill,red] at (6*1.62+4,0) {};
\node[draw,circle,inner sep=1pt,fill] at (7*1.62+4,0) {};
\node[draw,circle,inner sep=1pt,fill] at (8*1.62+4,0) {};
\node[draw,circle,inner sep=1pt,fill] at (8*1.62+5,0) {};
\node[draw,circle,inner sep=1pt,fill] at (9*1.62+5,0) {};
\node[draw,circle,inner sep=1.5pt,fill,red] at (9*1.62+6,0) {};
\node[draw,circle,inner sep=1pt,fill] at (10*1.62+6,0) {};
\node[draw,circle,inner sep=1pt,fill] at (11*1.62+6,0) {};
\node[draw,circle,inner sep=1pt,fill] at (11*1.62+7,0) {};
\node[draw,circle,inner sep=1pt,fill] at (12*1.62+7,0) {};
\node[draw,circle,inner sep=1.5pt,fill,red] at (12*1.62+8,0) {};
\node[draw,circle,inner sep=1pt,fill] at (13*1.62+8,0) {};
\node[draw,circle,inner sep=1pt,fill] at (14*1.62+8,0) {};
\node[draw,circle,inner sep=1pt,fill] at (-0*1.62-1,0) {};
\node[draw,circle,inner sep=1pt,fill] at (-1*1.62-1,0) {};
\node[draw,circle,inner sep=1pt,fill] at (-2*1.62-1,0) {};
\end{tikzpicture}
  \caption{An arithmetic progression of length $n=5$ in the Fibonacci model set. The common ratio is $r=3\tau+2$, with conjugate $r' \simeq 0.1459$. Note that $-1 \leq 0, r', 2r', 3r', 4r' <\tau-1$.}
\end{figure}

\medskip

Next, we show that any element of the Fibonacci can be the first term of such a sequence, and for a fixed element in Fibonacci we list all the possible values of the common difference $t$.

\begin{lemma}\label{P1} Let $n \in \NN$, and let $\Lambda=\oplam([-1, \tau-1))$ be the Fibonacci model set. Let $s \in \Lambda$ be arbitrary. Then,
\begin{itemize}
    \item[(a)] There exists some $t \in \RR$ such that
\begin{displaymath}
s, s+t, s+2t,..., s+nt \in \Lambda \,.
\end{displaymath}
\item[(b)] For each $s \in \Lambda$ we have
\begin{displaymath}
\{ t : s, s+t, s+2t,..., s+nt \in \Lambda \} =\oplam( [\frac{-1-s^\star}{n} , \frac{\tau-1-s^\star}{n})) \,.
\end{displaymath}
In particular, this set is relatively dense.
\end{itemize}
\end{lemma}

\begin{proof}

First note that, to have $s, s+t \in \Lambda$, we must have $t \in \Lambda - \Lambda \subset L$.

Next, for all $s \in \Lambda, t \in L$ we have by the definition of $\Lambda$:
\begin{align*}
s, s+t, s+2t,..., s+nt \in \Lambda &\Leftrightarrow s^\star, s^\star+t^\star,...., s^\star+nt^\star \in [-1 ,\tau-1) \\
&\Leftrightarrow  -1 \leq s^\star, s^\star+t^\star,...., s^\star+nt^\star < \tau-1  \\
&\Leftrightarrow  \frac{-1-s^\star}{n} \leq t^\star < \frac{\tau-1}{n}  -\frac{s^\star}{n}  \\
\end{align*}

Since the interval $[\frac{-1-s^\star}{n} , \frac{\tau-1-s^\star}{n} )$ has non-empty interior, the claims follow now from Lemma~\ref{model sets are rel dense}.
\end{proof}

\medskip

Next, we provide a stronger version of Proposition~\ref{P1}. This will allow us prove a van der Waerden type theorem for the Fibonacci model set.

\begin{lemma}\label{L1} For each $n \in \NN$ there exists some $R>0$ such that, for all $x \in \RR$ the set
$$\Lambda \cap [x, x+R]$$
contains a nontrivial arithmetic progression of length $n$.
\end{lemma}
\begin{proof}

Since $[-1, \frac{\tau}{2}-1)$ has non-empty interior, the set $\oplam([-1, \frac{\tau}{2}-1))$ is relatively dense by Lemma~\ref{model sets are rel dense}. Therefore, there exists some $A>0$ such that $\oplam([-1, \frac{\tau}{2}-1))+[-A,0]=\RR$.

Next, as $(0, \frac{\tau}{2n})$ has non-empty interior, the set $\oplam((0, \frac{\tau}{2n}))$ is relatively dense by Lemma~\ref{model sets are rel dense}. Therefore, there exists some $t \in \oplam((0, \frac{\tau}{2n})) \cap (0, \infty)$.

Let $R:=A+nt$. We claim that this $R$ has the desired property.

Indeed, let $x \in \mathbb R$. Then, since $\oplam([-1, \frac{\tau}{2}-1))+[-A,0]=\RR$, there exists some $s \in  \oplam([-1, \frac{\tau}{2}-1)) \cap [x, x+A]$.

Then, we have for all $0 \leq k \leq n$
\begin{align*}
x&\leq  s < s+t< ...<  s+nt \leq x+A+nt=x+R  \,,\\
s^\star +kt^\star &\in [-1, \frac{\tau}{2}-1)+ (0, \frac{k\tau}{2n}) \subset [-1, \tau-1) \,.\\
\end{align*}

This gives
\[
s, s+t,..., s+nt \in \Lambda \cap [x, x+R] \,.
\]
\end{proof}

\smallskip

Next, let us prove the following simple fact.

\begin{fact}\label{R1} In the Fibonacci CPS, if $a<b$ then $\oplam((a,b))$ is $\frac{\tau^3}{b-a}$ relatively dense.
\end{fact}
\begin{proof}
Let us observe first that for any interval $I$, since $\tau^2=\tau+1$ we have
\begin{align*}
\tau \oplam(I)&=\tau \{ m+n \tau : m,n \in \ZZ, m+n \tau' \in I\} \\
&= \{ m\tau+n \tau^2 : m,n \in \ZZ, m+n \tau' \in I \}\\
&= \{ m\tau+n \tau+n : m,n \in \ZZ, \tau'(m+n \tau') \in \tau' I \}\\
&= \{ (m+n) \tau+n : m,n \in \ZZ, \tau'(m+n)+n \in \tau' I\}\subseteq \oplam(\tau'I) \,.
\end{align*}

From here, it follows immediately by induction that
\begin{displaymath}
\tau^{2n} \oplam([-1, \tau-1)) \subseteq \oplam\left([-(\tau')^{2n}, (\tau')^{2n}(\tau-1))\right) \,.
\end{displaymath}

Since $\oplam([-1, \tau-1))$ is $\tau$-relatively dense, we get that $\oplam([-(\tau')^{2n}, (\tau')^{2n}(\tau-1)))$ is $\tau^{2n}$-relatively dense.

Note that $[-(\tau')^{2n}, (\tau')^{2n}(\tau-1))$ is an interval of length $|\tau'|^{2n-1}$.

Next, pick any $a <b$ in $\RR$. Pick $n \in \ZZ$ such that
\begin{displaymath}
|\tau'|^{2n-1}< b-a\leq |\tau'|^{2n-3} \,.
\end{displaymath}
This is possible as $\lim_{n \to \infty} |\tau'|^{2n-1}=0$ and $\lim_{n \to -\infty}|\tau'|^{2n-1}=\lim_{n \to \infty}|\tau|^{2n-1}=\infty \,.$

Since $|\tau'|^{2n-1}< b-a$ and $L^\star$ is dense in $\mathbb R$, there exists some $(t,t^\star)\in L$ such that $t^\star+[-(\tau')^{2n}, (\tau')^{2n}(\tau-1)) \subseteq (a,b)$. Then
\begin{displaymath}
\oplam([-(\tau')^{2n}, (\tau')^{2n}(\tau-1))) \subseteq \oplam(t^\star+(a,b))=-t+  \oplam((a,b)) \,.
\end{displaymath}
Since $\oplam([-(\tau')^{2n}, (\tau')^{2n}(\tau-1)))$ is $\tau^{2n}$-relatively dense, it follows that
$\oplam((a,b))$ is $\tau^{2n}$-relatively dense.

Finally $b-a\leq \tau'^{2n-3}$ gives
\begin{displaymath}
\tau^{2n}=\frac{\tau^3}{|\tau'|^{2n-3}} \leq \frac{\tau^3}{b-a} \,.
\end{displaymath}

This proves the fact.
\end{proof}

\begin{remark} Using Fact~\ref{R1} we can give an upper-bound on $R$ in Lemma~\ref{L1}.

Indeed, in the proof of Lemma~\ref{L1}, by Fact~\ref{R1} we can chose $A=2 \tau^2 $ and $0< t \leq 2n \tau^2  $, thus
\begin{displaymath}
R = 2 \tau^2 + 2n^2 \tau^2=2(n^2+1) \tau^2
\end{displaymath}
works.
\end{remark}
\medskip

As an immediate consequence of Lemma~\ref{L1} we get the following result.

\begin{theorem}[van der Waerden's theorem for Fibonacci]\label{TV1} For any given positive integers $r$ and $k$, there is some $R>0$ such that, for each coloring of the points of $\Lambda$ with $r$ colors, and for each $x\in \RR$ the set $\Lambda \cap [x, x+R]$ contains $k$ elements of the same color in arithmetic progression.
\end{theorem}
\begin{proof} Let $N$ be so that the van der Waerden's theorem (Theorem~\ref{van der W}) holds for $r, k$ applied to $\{ 1,2,...N\}$.

By Lemma~\ref{L1}, there exists some  $R>0$ such that, for all $x \in \RR$ the set
$$\Lambda \cap [x, x+R]$$
contains a nontrivial arithmetic progression of length $n$.

We claim that this $R$ works.

Consider any coloring of $\Lambda$ with $r$ colors. Let $x \in \RR$ be arbitrary.

By Lemma~\ref{L1}, we can find $s,t$ such that, $t \neq 0$, and for all $1 \leq k \leq N$ we have
\[
a_j:= s+jt \in \Lambda \cap [x,x+A] \,.
\]

Now, color the set $\{ 1, 2,... ,N \}$ the following way: color $j$ with the color of $a_j$. By Theorem~\ref{van der W}, there exists an arithmetic progression $l_1<l_2<..<l_k$ of length $k$ of the same color.

Then, $s+l_1t, s+l_2t,..., s+l_kt \in  \Lambda \cap [x, x+A]$ are $k$ elements of the same color in arithmetic progression.
\end{proof}

\section{Arithmetic progressions in model sets}

In this section we show that model sets in $\RR^d$ contain arbitrary long arithmetic progressions, and prove a van der Waerden type result for model sets.

\smallskip

\begin{tcolorbox}
For this entire section $(\RR^d, H, \cL)$ is a fixed CPS and $\Lambda=\oplam(W)$ is a fixed model set in this CPS.
\end{tcolorbox}

\smallskip

Next we prove the following preliminary lemma which we will need in this section.

\begin{lemma}\label{U rel dense} Let $(\RR^d,H, \cL)$ be a CPS and let $0 \in U \subset H$ be open. Then $\oplam(U) \backslash \{ 0 \}$ is relatively dense.
\end{lemma}
\begin{proof}

First, the set $\oplam(U)$ is relatively dense by Lemma~\ref{model sets are rel dense}. Therefore, there exists some $R'>0$, such that $\RR^d=\oplam(U)+B_{R'}(0)$. In particular, there exists some $y \in \oplam(U) \backslash \{ 0\}$.

Fix such $y$, and define $R=R'+d(y,0)$. We claim that $\oplam(U) \backslash \{ 0\}$ is $R$-relatively dense.

Indeed, let $z \in \RR^d$. Then $z \in \oplam(U)+B_{R'}(0)$, and hence, there exists some $a \in \oplam(U)$ such that $d(z,a) < R'$.

We split the proof now in two cases.

\begin{itemize}
    \item{} $a \neq 0$.
Then $a \in \oplam(U) \backslash \{ 0\}$ and $d(z,a) < R' <R$.
    \item{} $a=0$.
In this case we have  $y \in  \oplam(U) \backslash \{ 0\}$ and $d(z,y)  \leq d(z,0)+d(0,y) <R'+d(0,y)=R$.
\end{itemize}

Therefore, in both cases there exists some $x \in  \oplam(U) \backslash \{ 0\}$ such that $d(z,x) <R$.

This completes the proof.
\end{proof}

\medskip

Now, we can prove that model sets have arbitrary long arithmetic progressions.

\begin{proposition}\label{P2} Let $n \in \NN$. Then, for each $s \in \oplam(\mbox{Int}(W))$ there exists an open set $0 \in U \subset H$ such that, $\oplam(U) \backslash \{ 0\}$ relatively dense in $\RR^d$, and for each $t \in \oplam(U) \backslash \{ 0 \}$ we have
\begin{displaymath}
s, s+t, s+2t,..., s+nt \in \Lambda=\oplam(W) \,.
\end{displaymath}

\end{proposition}
\begin{proof}
 Since $s^\star \in \mbox{Int}(W)$, which is open, we can find an open set $U$ such that
\begin{displaymath}
s^\star+\underbrace{U+U+..+U}_{n \mbox{ times }} \subset  \mbox{Int}(W) \,.
\end{displaymath}

From here, it follows immediately that
\begin{equation}\label{eq21}
s, s+t, s+2t,..., s+nt \in \Lambda \qquad \forall t \in \oplam(U) \,.
\end{equation}

The claim follows now from \eqref{eq21} and Lemma~\ref{U rel dense}.

\end{proof}

\medskip

Next, we show that for each $n$ we can find arithmetic progressions of length $n$ within bounded gaps.

\begin{lemma}\label{L2} For  each $n \in \NN$ there exists some $R>0$ such that, for all $x \in \RR^d$ the set
$$\Lambda \cap B_R(x)$$
contains a nontrivial arithmetic progression of length $n$.

\end{lemma}
\begin{proof}

Since $W$ has non-empty interior, we can find an open set $V \neq \emptyset$ and some open $0 \in U$ such that

\begin{displaymath}
V+\underbrace{U+U+..+U}_{n \mbox{ times }} \subset  W \,.
\end{displaymath}

Since $V$ is open, there exists some $R'>0$ such that $\oplam(V)+B_{R'}(0)=\RR^d$.

Since $0\in U$ is open, by Lemma~\ref{U rel dense} there exists some $R''>0$ such that $\left(\oplam(U)\backslash \{ 0\} \right)+B_{R''}(0)=\RR^d$.

Let $R:=R'+nR''$. We show that this $R$ works.

Let $x \in \RR^d$ be arbitrary. Then, there exists some $s \in \oplam(V) \cap B_{R'}(x)$.

Let $t \in \left(\oplam(U) \backslash \{0 \}\right) \cap B_{R''}(0)$. Then, since $0 \in U$, for all $0 \leq k \leq n $ we have
\begin{displaymath}
\underbrace{t^\star+t^\star+..+t^\star}_{k \mbox{ times }}+\underbrace{0+0+..+0}_{n-k \mbox{ times }} \in \underbrace{U+U+..+U}_{n \mbox{ times }}
\end{displaymath}
Therefore, for all $0 \leq k \leq n$ we have
\begin{displaymath}
s^\star+kt^\star \in V+\underbrace{U+U+..+U}_{n \mbox{ times }} \subset  W \Rightarrow s+kt \in \oplam(W)=\Lambda \,.
\end{displaymath}

Moreover, for all $0 \leq k \leq n$ we have
\begin{align*}
d(s+kt,x) &\leq d(s+kt, s)+d(s,x) = d(kt, 0) + d(s,x) \leq  k d(t,0)+ d(s,x) \\
&< k R''+R' \leq nR''+R'=R \,.
\end{align*}

This gives
\begin{displaymath}
s, s+t, s+2t, \dots , s+nt \in \Lambda \cap B_R(x) \,.
\end{displaymath}
\end{proof}

We now can prove:
\begin{theorem}[van der Waerden's theorem for model sets]\label{TV2} Let $\Lambda \subset \RR^d$ be a model set.

Then, for any given positive integers $r$ and $k$, there is some $R>0$ such that, if the points of $\Lambda$ are colored, each with one of $r$ different colors, then, for each $x\in \RR^d$ the set $\Lambda \cap B_R(x)$ contains a nontrivial arithmetic progression of length $k$, with all elements of the same color.
\end{theorem}
\begin{proof} Let $N$ be so that the van der Waerden's theorem (Theorem~\ref{van der W}) holds for $r,k$ applied to $\{ 1,2,...N\}$.

By Lemma~\ref{L2} there exists some compact $R$ such that, for all $x \in \RR^d$ the set
$$\Lambda \cap B_R(x)$$
contains a nontrivial arithmetic progression of length $N$.

We claim that this $R$ works.

Consider any coloring of $\Lambda$ with $r$ colors. Let $x \in \RR^d$ be arbitrary.

By Lemma~\ref{L1}, we can find $s,t$ such that, $t \neq 0$ and for all $1 \leq k \leq N$
\[
a_k:= s+kt \in \Lambda \cap B_R(x) \,.
\]
Moreover, the elements $a_k$ are pairwise distinct.

Now, color the set $\{ 1, 2,... ,N \}$ by coloring $k$ with the color of $a_k$. By van der Waerden's theorem, there exists an arithmetic progression $l_1<l_2<..<l_k$ of length $k$ of the same color.

Then, $s+l_1t, s+l_2t,..., s+l_kt \in  \Lambda \cap B_R(x)$ are $k$ elements of the same color in arithmetic progression.
\end{proof}

\section{Arithmetic progressions in Meyer sets}\label{sect:mey}

In this section we show that any Meyer set $\Lambda \subset \RR^d$ has arbitrary long arithmetic progressions.

\smallskip

\begin{tcolorbox}
For this entire section $\Lambda\subset \RR^d$ is a fixed Meyer set.
\end{tcolorbox}

\smallskip

\begin{theorem}\label{T1} Let $\Lambda\subset \RR^d$  be a Meyer set. Then, for each positive integer $k$, there exists some $R>0$, such that, for all $x \in \RR^d$ the set $\Lambda \cap B_R(x)$ contains a nontrivial arithmetic progression of length $k$.
\end{theorem}
\begin{proof} Since $\Lambda$ is a Meyer set, by Theorem~\ref{Meyer characterisation} there exists a model set $\oplam(W)$ such that $\Lambda \subset \oplam(W)$.

By \cite[Lemma~5.5.1]{NS11} there exists a finite set $F=\{ t_1,.., t_r \}$ such that
\begin{displaymath}
\oplam(W) \subset \Lambda +F \,.
\end{displaymath}

Let $R'$ be the constant given by Theorem~\ref{TV2} for $\oplam(W)$, $r$ colors and arithmetic progressions of length $k$. Let
\begin{displaymath}
R:= \max \{ R'+d(t_j,0) : 1 \leq j \leq r \} \,.
\end{displaymath}
We show that this $R$ works.

We color $\oplam(W)$ with $r$ colors the following way: Since $\oplam(W) \subset \Lambda +F$, for each $x \in \oplam(W)$ there exists some $1 \leq j \leq r$ such that $x \in t_j+\Lambda$. Then, color each $x \in \oplam(W)$ by the color $\min\{ j: x \in t_j+\Lambda\}$. We use the minimum since some $x\in \oplam(W)$ may belong to multiple sets $t_j +\Lambda$, in which case we need to make a choice (any choice here makes the rest of the proof work).

Let $x \in \RR^d$ be arbitrary.

Now, by Theorem~\ref{TV2}, there exists a monochromatic nontrivial arithmetic progression $a_1,..,a_k$ of length $k$ in $\oplam(W) \cap B_{R'}(x)$. By our construction of the coloring, there exists some $j$ such that
\begin{displaymath}
a_1,..,a_k \in t_j +\Lambda \,.
\end{displaymath}

It follows that $a_1-t_j, a_2-t_j,.., a_k-t_j \in \Lambda$ is a non-trivial arithmetic progression of length $k$. Moreover, for each $1 \leq i \leq k$ we have
\begin{displaymath}
d(a_k-t_j,x) \leq d(a_k,x)+d(-t_j,0) <  R'+d(t_j,0)  \leq R \,.
\end{displaymath}
\end{proof}

\smallskip

As an immediate consequence we get the following.
\begin{theorem}[van der Waerden's theorem for Meyer sets]\label{TV5}  Let $\Lambda\subset \RR^d$  be a Meyer set. For any given positive integers $r$ and $k$, there is some $R>0$ such that, if the points of $\Lambda$ are colored, each with one of $r$ different colors, then, for each $x\in \RR^d$ the set $\Lambda \cap B_R(x)$ contains $k$ elements of the same color in arithmetic progression.
\end{theorem}
\begin{proof} Let $N$ be so that the van der Waerden's theorem (Theorem~\ref{van der W}) holds for $r,k$ applied to $\{ 1,2,...N\}$.

By Theorem~\ref{T1} there exists some $R'$ such that, for all $x \in \RR^d$ the set
$$\Lambda \cap B_{R'}(x)$$
contains a nontrivial arithmetic progression of length $N$.

The rest of the proof is identical to the one of Theorem~\ref{TV2}.
\end{proof}

\section{Pure point diffractive sets}\label{sect:pure}

Now we can prove the following results, which is a weak Szemerédi's type theorem.

\begin{theorem}\label{P6} Let $\Lambda \subset \RR^d$ be point set, let $A_n=[-n,n]^d$ and $\cB$ be a subsequence of $\cA:= \{ A_n \}_n$. Assume that
\begin{itemize}
    \item[(i)] $\Lambda - \Lambda$ is uniformly discrete.
    \item[(ii)] $\overline{\dens_{\cB}}(\Lambda)>0$.
    \item[(iii)] $\Lambda$ is pure point diffractive with respect to $\cB$.
\end{itemize}

Then, for each $\epsilon >0$ and each $n \in \NN$, there exists a relatively dense set $T \subset \RR^d$ such that, for each $t \in T$ we have
\begin{displaymath}
\ud(\Lambda \cap (t+\Lambda) \cap ... \cap (nt+\Lambda)) \geq (1-\epsilon) \ud(\Lambda) \,.
\end{displaymath}
In particular, $\Lambda$ has arithmetic progressions of arbitrary length.
\end{theorem}
\begin{proof}

Let $\epsilon >0$ and $n \in \NN$ be fixed. Then, by Proposition~\ref{P4} the set
\[
T=\{ t \in \RR^d : d(\Lambda, t+\Lambda) < \frac{\epsilon \ud(\Lambda)}{n} \}
\]
is relatively dense.

Let $t \in T$ be arbitrary, and let $\Gamma_t:=\Lambda \cap (t+\Lambda) \cap ... \cap (nt+\Lambda)$.

A simple computation shows that
\begin{equation}\label{eq232}
\begin{split}
\Lambda \backslash \Gamma_t &\subseteq \left( \Lambda \Delta (t+\Lambda) \right) \cup \left( (t+\Lambda) \Delta (2t+\Lambda) \right) \cup \\
&... \cup \left( ((n-1)t+\Lambda) \Delta (nt+\Lambda) \right) \,.
\end{split}
\end{equation}

Indeed, if $s \in \Lambda \backslash \Gamma_t$, then $s \notin \Lambda \cap (t+\Lambda) \cap ... \cap (nt+\Lambda)$. Let $k$ be the smallest $0 \leq k \leq n$
such that $s \notin kt+\Lambda$.

Since $s \in \Lambda$ we have $k \geq 1$ and $s \in (k-1)t+\Lambda$. This shows that
\begin{align*}
s \in ((k-1)t+\Lambda) \Delta (kt+\Lambda) &\subseteq \left( \Lambda \Delta (t+\Lambda) \right) \cup \left( (t+\Lambda) \Delta (2t+\Lambda) \right) \\
&\cup \dots \cup \left( ((n-1)t+\Lambda) \Delta (nt+\Lambda) \right)
\end{align*}

Now, since $\Gamma_t \subset \Lambda$ we have
\begin{align*}
\ud(\Lambda \backslash \Gamma_t) &= \ud (\Lambda \Delta \Gamma_t) = d(\Lambda, \Gamma_t) \stackrel{\eqref{eq232}}{\leq} \sum_{k=0}^{n-1} d( kt+\Lambda, (k+1)t+\Lambda) \\
&\stackrel{\mbox{Lemma}~\ref{L5}}{=\joinrel=\joinrel=\joinrel=\joinrel=\joinrel=\joinrel=\joinrel=\joinrel=\joinrel=\joinrel=\joinrel=\joinrel=} \sum_{k=0}^{n-1} d( \Lambda, t+\Lambda) < n \frac{\epsilon \ud(\Lambda)}{n} \,.\\
\end{align*}

Therefore, as $\Lambda \subset \Gamma_t \cup (\Lambda \backslash \Gamma_t)$ we have
\begin{displaymath}
\ud(\Lambda) \leq \ud(\Gamma_t) + \ud(\Lambda \backslash \Gamma_t) \leq \ud(\Gamma_t) +  \epsilon \ud(\Lambda) \,,
\end{displaymath}
which gives
\begin{displaymath}
\ud(\Lambda \cap (t+\Lambda) \cap ... \cap (nt+\Lambda)) \geq (1-\epsilon) \ud(\Lambda) \,.
\end{displaymath}

Finally, for a fixed $0 < \epsilon <1$, $T$ is relatively dense and hence infinite. Then, we can pick some $0 \neq t \in T$ and we get that
\begin{displaymath}
\ud(\Lambda \cap (t+\Lambda) \cap ... \cap (nt+\Lambda)) >0 \,,
\end{displaymath}
which gives that $\Lambda \cap (t+\Lambda) \cap ... \cap (nt+\Lambda) \neq \emptyset$. Picking some $s \in \Lambda \cap (t+\Lambda) \cap ... \cap (nt+\Lambda) $ we then get
\begin{displaymath}
s, s-t, s-2t, \dots , s-nt \in \Lambda \,,
\end{displaymath}
which proves the claim.

\end{proof}

\begin{remark} Under the assumptions of Theorem~\ref{P6}, as the autocorrelation $\gamma$ of $\Lambda$ exists with respect to $\cB$, the density of $\Lambda$ exists with respect to $\cB$. Moreover, the density is non-zero exactly when the diffraction is non-trivial.
\end{remark}

\bigskip

Given a CPS $(\RR^d ,H, \cL)$ and a compact window $W$, recall that $\oplam(W)$ is called a \textbf{weak model set of maximal density} with respect to a subsequence $\cB$ of $\cA=\{ [-n,n]^s\}_n$, if
\begin{displaymath}
\lim_n \frac{1}{\vol(B_n)} \mbox{card} \bigl(\oplam(W) \cap B_b \bigr) = \mbox{dens}(\cL) \theta_H(W) \,.
\end{displaymath}
For an overview of maximal density model sets and their properties see \cite{BHS,KR}.

\smallskip

As an immediate consequence of Theorem~\ref{P6} we get the following result.

\begin{corollary} Let $(\RR^d ,H, \cL)$ be a CPS and let $W \subset H$ be compact and let $\cB$ be subsequence of $A_n=[-n,n]^d$. If $\oplam(W)$ has maximal density with respect to $\cB$, and if $\theta_H(W) \neq 0$ then $\oplam(W)$ contains arithmetic progressions of arbitrary length.
\end{corollary}
\begin{proof}
Since $\oplam(W)$ has maximal density with respect to $\cB$ and  $\theta_H(W) \neq 0$, $\oplam(W)$ has positive density. Moreover $\oplam(W) - \oplam(W) \subset \oplam(W-W)$ is uniformly discrete as $W-W$ is compact.

Finally $\oplam(W)$ is pure point diffractive with respect to $\cB$ by \cite{BHS,KR}.

The claim follows now from Theorem~\ref{P6}.
\end{proof}

By combining this with \cite{RVM} we get:

\begin{corollary} Let $(\RR^d ,H, \cL)$ be a CPS, let $W \subset H$ be compact and let $A_n=[-n,n]^d$. If $\theta_H(W) \neq 0$ then $\oplam(t+W)$ contains arithmetic progressions of arbitrary length for almost all $(s,t) \in \TT:= (\RR^d \times H)/\cL$.
\end{corollary}

\smallskip

We complete the section by introducing a class of pure point diffractive Delone sets which don't contain arithmetic progressions of length 3.

\begin{proposition} Let $\{ a_n \}_{n \in \ZZ} \in \mathbb R$ be such that
\begin{itemize}
    \item[a)]$\lim_{n \to \pm \infty} a_n = 0$.
    \item[b)] $\ldots ,a_{-n} ,\ldots ,a_{-2}, a_{-1}, 1, a_0, a_1, a_2, \ldots , a_n, \ldots$ are linearly independent over $\mathbb Q$.
\end{itemize}
Define
\begin{displaymath}
\Lambda := \{ n+ a_n : n \in \ZZ \}
\end{displaymath}
Then, $\Lambda$ is pure point diffractive and does not contain three elements in arithmetic progression.
\end{proposition}
\begin{proof}

First, since $a_n$ are linearly independent, we have $n+a_n \neq m+a_m$ for all $n \neq m$.
Indeed, if we assume by contradiction that there exists some $n\neq m$ such that $n+a_n=m+a_m$ then
\begin{displaymath}
(m-n)\cdot 1 +1 \cdot a_m +(-1)\cdot a_n =0 \,,
\end{displaymath}
contradicting the linear independence.

It follows that $\delta_\Lambda - \delta_{\ZZ}$ is a measure vanishing at infinity, and hence null weakly almost periodic \cite{SpSt}.

Therefore $\Lambda$ is pure point diffractive \cite{LS}. More generally, $\Lambda$ has the same diffraction as $\ZZ$ \cite{LS}, hence
\begin{displaymath}
\widehat{\gamma}=\delta_\ZZ \,.
\end{displaymath}

The claim of not containing three points in arithmetic progression is an immediate consequence of the linear independence. Indeed, if $(n+a_n), (m+a_m), (k+a_k)$ are three distinct elements in arithmetic progression then
\begin{align*}
2(m+a_m)&=(n+a_n) +(k+a_k) \\
0&=(n+k-2m)\cdot 1+ 1 \cdot a_n+1 \cdot a_k -2 \cdot a_m \\
\end{align*}
Since $m,n,k$ are distinct, this contradicts the linear independence of $1, a_m, a_n, a_k$.
\end{proof}

\begin{example} An explicit such example is
\begin{displaymath}
\Lambda= \{ n + \frac{e^{2n+2}}{3^{2n+2}} : n \in \NN \} \cup \{ -n + \frac{e^{2n+1}}{3^{2n+1}} : n \in \NN \} \cup \{0 \} \,.
\end{displaymath}
\end{example}

\smallskip

Let us mention here that we suspect that the full Szemer\'edi Theorem is true in model sets, but the proof is probably very long and tedious, and beyond the scope of this paper. We state this here as a conjecture:

\begin{conj}[Szemer\'edi Conjecture for model sets] Let $\oplam(W)$ be a model set in a CPS $(G, H, \cL)$ with non-compact and torsion free $G$, and let $\cA$ be a van Hove sequence in $G$.

If $\Lambda \subset \oplam(W)$ is a subset such that
\begin{displaymath}
\ud_{\cA} (\Lambda) >0 \,,
\end{displaymath}
then $\Lambda$ contains arithmetic progressions of arbitrary length.
\end{conj}

\subsection*{Acknowledgments}  The work was partially supported by NSERC with grant 03762-2014 (NS) and 2019-07097 (AT), and the authors are thankful for the support. AK was partially supported by Dr. Chris Ramsey via NSERC grant no 2019-05430 and she is grateful for the support.

\end{document}